\documentclass[11 point, a4paper, reqno]{amsart}
\usepackage[english]{babel}
\usepackage{amsthm}
\usepackage{graphicx}
\usepackage{amsfonts}
\usepackage{eufrak}
\usepackage{amsopn}
\usepackage{amsmath, amsfonts, amssymb, amscd, amsthm}
\usepackage{yhmath}
\usepackage{setspace}
\usepackage{mathtools}
\usepackage{multirow}
\usepackage{caption}
\usepackage{graphicx}
\usepackage[all]{xy}
\usepackage{lscape}
\usepackage{comment}
\usepackage{float}
\usepackage{fancyhdr}

\newcommand{\Stirling}[2]{\genfrac{[}{]}{0pt}{0}{#1}{#2}}

\newtheorem{theorem}{Theorem}[section]
\newtheorem{lemma}[theorem]{Lemma}
\newtheorem{proposition}[theorem]{Proposition}

\theoremstyle{definition}

\theoremstyle{remark}
\newtheorem{remark}[theorem]{Remark}
\theoremstyle{example}

\theoremstyle{note}

\numberwithin{equation}{section}

\DeclareMathOperator{\GL}{GL}

\DeclareMathOperator{\Span}{Span}

\DeclareMathOperator{\M}{M}

\DeclareMathOperator{\tr}{Tr}

\DeclareMathOperator{\Rank}{Rank}

\DeclareMathOperator{\Res}{Res}

\begin{document}
\title{Dimension formula for the twisted Jacquet module of a cuspidal representation of $\GL(2n,\mathbb{F}_q)$}
\author{Kumar Balasubramanian}

\thanks{Research of Kumar Balasubramanian is supported by the SERB grant: CRG/2023/000281.}
\address{Kumar Balasubramanian\\
Department of Mathematics\\
IISER Bhopal\\
Bhopal, Madhya Pradesh 462066, India}
\email{bkumar@iiserb.ac.in}

\author{Himanshi Khurana}
\thanks{Research of Himanshi Khurana is supported by HRI postdoctoral fellowship.}
\address{Himanshi Khurana\\
Department of Mathematics\\
Harish-Chandra Research Institute\\
Prayagraj, Uttar Pradesh 211019 India}
\email{himanshikhurana@hri.res.in}

\keywords{Cuspidal representations, Twisted Jacquet module}
\subjclass{Primary: 20G40}

\maketitle

\begin{abstract} Let $F$ be a finite field and $G=\GL(2n,F)$. In this paper, we calculate the dimension of the twisted Jacquet module $\pi_{N,\psi_{A}}$ where $A\in \M(n,F)$ is a rank $k$ matrix and $\pi$ is an irreducible cuspidal representation of $G$.
\end{abstract}

\section{Introduction}
\thispagestyle{empty}
Let $F$ be a finite field and $G=\GL(2n,F)$. Let $P=MN$ be the standard parabolic subgroup of $G$ corresponding to the partition $(n,n)$. Then, we have that $M \simeq \GL(n,F) \times \GL(n,F)$ and $N \simeq \M(n,F)$. Let $(\pi,V)$ be an irreducible cuspidal representation of $G$. Let $\psi_0$ be a non-trivial additive character of $F$. It is easy to see that any character of $\M(n,F)$ is of the form $\psi_A$, where $\psi_A: \M(n,F) \to \mathbb{C}^{\times}$ is a character given by 
\[\psi_A(X)=\psi_0(\tr(AX)).\]
A character $\psi_A$ of $\M(n,F)$ is said to be non-degenerate if $A \in \GL(n,F)$, while it is said to be a degenerate character of rank $k$ if $\Rank(A)=k<n$. The group $\GL(n,F) \times \GL(n,F)$ acts on the set of characters of $\M(n,F)$, via
\[(g_1,g_2) \cdot \psi_A=\psi_{{g_2}^{-1}Ag_1}.\]
For $1 \leq k \leq n$, consider
\[A_k=\begin{bmatrix}I_k & 0 \\
0 & 0 \end{bmatrix} \in  \M(n,F),\]
where $I_k$ is the identity matrix in $\GL(k,F)$ and $I_0=0$. The characters $\psi_{A_k}, 0 \leq k \leq n$ form a set of representatives for the orbits uder the above action. The character $\psi_{A_n}$ is a representative for the orbit of non-degenerate characters, while the character $\psi_{A_k}, k < n$ is a representative for the orbit of degenerate characters of rank $k$.\\

One of the fundamental questions in representation theory is to understand the decomposition of a representation into its irreducible subrepresentations and the multiplicity with which these components occur. Consider the restriction of the representation $\pi$ to $N$, $\Res_{N}^{G} \pi$. In this paper, we give an explicit formula for the multiplicities with which any character occurs inside $\Res_N^G\pi$. Since $N$ is abelian, 
\[\Res_N^G \pi= \bigoplus_{A \in \M(n,F)} m_{\psi_A}\psi_A,\]
where $m_{\psi_A}$ is the multiplicity with which $\psi_A$ occurs in $\Res_N^G \pi.$ The sum of all irreducible representations of $N$ inside $\pi$, on which $N$ operates by $\psi_{A}$ is called the twisted Jacquet module of $\pi$ corresponding to the character $\psi_{A}$, and is denoted by $\pi_{N,\psi_{A}}$. In other words, $\pi_{N,\psi_A}$ is the $\psi_A$-isotypic component of the representation $\Res_N^G \pi$. Then we have that \[\dim_{\mathbb{C}}({\pi_{N,\psi_A}})=m_{\psi_A}\dim_{\mathbb{C}}(\psi_A).\]
From Theorem 3.9 in \cite{KumHim[3]}, it follows that for any two matrices $A,B \in \M(n,F)$ of equal ranks, we have 
\[\dim_{\mathbb{C}}(\pi_{N,\psi_A})=\dim_{\mathbb{C}}(\pi_{N,\psi_B}).\]
Thus, we can deduce that
\[\Res_N^G \pi= \bigoplus_{k=0}^{n} \bigoplus_{A \in \M(n,k,F)}m_{\psi_{A_k}} \psi_{A_k}.\]
 Since $\pi$ is cuspidal, the Jacquet module $\pi_N=\pi_{N,1}=0$. Thus, $m_{\psi_{A_0}}=0$. In \cite{Dip[1]}, Prasad computed the multiplicity of any non-degenerate character inside $\Res_N^G\pi$ by computing the dimension of the space $\pi_{N,\psi_{A_n}}$. In an earlier work of ours, we computed the multiplicity of a degenerate character of rank $1$ inside $\Res_N^G\pi$ by computing the dimension of $\pi_{N,\psi_{A_1}}$. See Theorem 3.7 in \cite{KumHim[3]} for more details. In \cite{KumHimgl6rk2}, we also computed the dimension of $\pi_{N,\psi_{A_2}}$ for $n=3$, and hence the multiplicity with which any degenerate character of rank $2$ occurs inside $\Res_N^G \pi$, when $\pi$ is a cuspidal representation of $\GL(6,F)$. See Theorem 3.1 in \cite{KumHimgl6rk2} for more details.\\
 
 In this paper, we give an explicit formula to describe the multiplicity of any degenerate character $\psi_{A_k}$ of rank $k, 1 < k < n$ by computing the dimension of the module $\pi_{N,\psi_{A_k}}$.

\section{Preliminaries}
In this section, we record some preliminaries that we need.

\subsection{Character of a cuspidal representation} Let $F$ be the finite field of order $q$ and $G=\GL(m,F)$. Let $F_m$ be the unique field extension of $F$ of degree $m$. A character $\theta$ of $F^{\times}_{m}$ is called a ``regular'' character, if under the action of the Galois group of $F_{m}$ over $F$, $\theta$ gives rise to $m$ distinct characters of $F^{\times}_{m}$. It is a well known fact that the cuspidal representations of $\GL(m,F)$ are parametrized by the regular characters of $F_{m}^{\times}$. To avoid introducing more notation, we mention below only the relevant statements on computing the character values that we have used. We refer the reader to Section 6 in \cite{Gel[1]} for more precise statements on computing character values.

\begin{theorem}\label{character value of cuspidal representation (Gelfand)}
Let $\theta$ be a regular character of $F^{\times}_{m}$. Let $\pi=\pi_{\theta}$ be an irreducible cuspidal representation of $\GL(m,F)$ associated to $\theta$. Let $\Theta_{\pi}$ be its character. If $g\in \GL(m,F)$ is such that the characteristic polynomial of $g$ is not a power of a polynomial irreducible over $F$. Then, we have  \[\Theta_\pi(g)=0. \]
\end{theorem}

\begin{theorem}\label{character value of cuspidal representation (Dipendra)} Let $\theta$ be a regular character of $F^{\times}_{m}$. Let $\i=\pi_{\theta}$ be an irreducible cuspidal representation of $\GL(m,F)$ associated to $\theta$. Let $\Theta_{\pi}$ be its character. Suppose that $g=s.u$ is the Jordan decomposition of an element $g$ in $\GL(m,F)$. If $\Theta_{\pi}(g)\neq 0$, then the semisimple element $s$ must come from $F_{m}^{\times}$. Suppose that $s$ comes from $F_{m}^{\times}$. Let $z$ be an eigenvalue of $s$ in $F_{m}$ and let $t$ be the dimension of the kernel of $g-z$ over $F_{m}$. Then
\[\Theta_{\pi}(g)=(-1)^{m-1}\bigg[\sum_{\alpha=0}^{d-1}\theta(z^{q^{\alpha}})\bigg ](1-q^{d})(1-(q^{d})^{2})\dots (1-(q^{d})^{t-1}). \]
where $q^{d}$ is the cardinality of the field generated by $z$ over $F$, and the summation is over the distinct Galois conjugates of $z$.
\end{theorem}

See Theorem 2 in \cite{Dip[1]} for this version.

\subsection{Twisted Jacquet module}

In this section, we recall the character and the dimension formula of the twisted Jacquet module of a representation $\pi$. \\

Let $G=\GL(k,F)$ and $P=MN$ be a parabolic subgroup of $G$. Let $\psi$ be a character of $N$. For $m\in M$, let $\psi^{m}$ be the character of $N$ defined by $\psi^{m}(n)=\psi(mnm^{-1})$. Let \[V(N,\psi)= \Span_{\mathbb{C}} \{\pi(n)v-\psi(n)v \mid n \in N, v \in V \} \]
and \[M_{\psi} = \{m \in M \mid {\psi}^{m}(n)=\psi(n) , \forall n \in N \} . \]
Clearly, $M_{\psi}$ is a subgroup of $M$ and it is easy to see that $V(N,\psi)$ is an $M_{\psi}$-invariant subspace of $V$. Hence, we get a representation  $(\pi_{N,\psi},V/V(N,\psi))$ of $M_{\psi}$. We call $(\pi_{N,\psi}, V/V(N,\psi))$ the twisted Jacquet module of $\pi$ with respect to $\psi$. We write $\Theta_{N, \psi}$ for the character of $\pi_{N, \psi}$.

\begin{proposition} \label{Character of twisted Jacquet module}
Let $(\pi,V)$ be a representation of $\GL(k,F)$  and $\Theta_\pi$ be the character of $\pi$. We have
\[\Theta_{N,\psi}(m) = \frac{1}{|N|}\sum_{n \in N}\Theta_{\pi}(mn)\overline{\psi(n)}.\]
\end{proposition}

\begin{remark}
Taking $m=1$, we get the dimension of $\pi_{N, \psi}$. To be precise, we have
\[\dim_{\mathbb{C}}(\pi_{N,\psi}) =  \frac{1}{|N|}\sum_{n \in N}\Theta_{\pi}(n)\overline{\psi(n)}.\]
\end{remark}

Let $n$ be a positive integer. Let $\M(n,F)$ be the set of $n \times n$ matrices over the finite field $F$. For an integer $0 \leq r \leq n$, we denote $\M(n,r,F)$ to be the subset of $\M(n,F)$ consisting of matrices of rank $r$. Denote the cardinality of $\M(n,r,F)$ by $a(n,r,q)$. For $0 \leq k \leq n$ and $\alpha \in F$, let 
\[Y_{n,r,k}^{\alpha}=\{X \in \M(n,r,F) \mid \tr(A_kX)=\alpha \}.\]
Denote the cardinality of $Y_{n,r,k}^{\alpha}$ by $f_{n,r,k}^{\alpha}$. The difference of the cardinalities $f_{n,r,k}^{0}-f_{n,r,k}^{1}$ is recorded in the theorem below. We refer the reader to Theorem 3.7 in \cite{KumKriHim} for more details.  

\begin{theorem}
For $0 \leq k,r \leq n$, we have
\[f_{n,r,k}^0-f_{n,r,k}^1=g_{n,r,k} = \sum_{i=0}^r 
(-1)^i q^{\tbinom{i}{2} +k(r-i)} \,\Stirling{k}{i}_q \,  a(n-k,r-i,q).
\]
\end{theorem}

Using change of variable $r-i=s$, the sum can be written as 
\[\sum_{i=0}^r 
(-1)^i q^{\tbinom{i}{2} +k(r-i)} \,\Stirling{k}{i}_q \,  a(n-k,r-i,q)=\sum_{s=0}^{r}(-1)^{r-s}q^{\tbinom{r-s}{2} +ks} \,\Stirling{k}{r-s}_q \,  a(n-k,s,q).\]
It is clear that $\Stirling{k}{r-s}_q=0$ when $r-s>k$. Therefore, it follows that

\begin{equation} \label{f(n,r,k)^0-f(n,r,k)^1}
f_{n,r,k}^0-f_{n,r,k}^1=\sum_{s=r-k}^{r}(-1)^{r-s}q^{\tbinom{r-s}{2} +ks} \,\Stirling{k}{r-s}_q \,  a(n-k,s,q)
\end{equation}

\noindent Before we proceed, we recall few identities that we use in this paper. Let $(x;q)_{n}$ be the $q$-Pochhammer symbol defined by 
\[(x;q)_{n}= \prod_{i=0}^{n-1}(1-xq^{i}).\] 
\begin{lemma} \label{identity for aq^k}
For any non-zero $a \in \mathbb{C}$, we have
\[(aq^k;q)_{n-k}=\frac{(a;q)_n}{(a;q)_k}.\]
\end{lemma}
\begin{proof}
By definition, we have that
\[
(aq^k;q)_{n-k}=(1-aq^k)(1-aq^{k+1})\cdots(1-aq^{n-1}).\]
On the other hand, 
\begin{align*} \frac{(a;q)_n}{(a;q)_k}&=\frac{\prod_{i=0}^{n-1}(1-aq^{i})}{\prod_{i=0}^{k-1}(1-aq^{i})}\\
&= (1-aq^k)(1-aq^{k+1})\cdots(1-aq^{n-1}).
\end{align*}
Hence the result.
\end{proof}

\noindent We also state $q$-analogue of the Chu-Vandermonde identity [\cite{gasper2004basic}, Eq.(1.5.2)]:
\begin{equation}
 \label{Chu Vandermode identity}
\sum_{r=0}^i \frac{\left(q^{-i} ; q\right)_r(b ; q)_r}{(c ; q)_r(q ; q)_r}\left(\frac{c q^i}{b}\right)^r=\frac{(c / b ; q)_i}{(c ; q)_i}
\end{equation}
where $i$ is a non-negative integer, and $b, c$ are complex numbers that satisfy $b \neq 0$ and $c \notin \left\{q^{-1}, \ldots, q^{-(i-1)}\right\}$. Before we proceed, we recall a $q$-hypergeometric identity (Lemma 2.7, \cite{HazGor}). 

\begin{proposition} 
\label{Hazan identity}
Let $n$ be a non-negative integer. Let $t$ be an integer greater than or equal to $2n$. Then
\[\sum_{r=0}^{n} a(n,r,q)(q;q)_{t-r}= q^{n^{2}}\frac{(q;q)^{2}_{t-n}}{(q;q)_{t-2n}}. \]
\end{proposition}

We refer the reader to Lemma 2.7 in \cite{HazGor} for a proof of the above proposition in a more general set up.

\section{Dimension of the Twisted Jacquet Module} 

Let $\pi=\pi_{\theta}$ be an irreducible cuspidal representation of $G$ corresponding to the regular character $\theta$ of $F_{2n}^{\times}$ and $\Theta_{\theta}$ be its character. In this section, we calculate the dimension of $\pi_{N,\psi_{A_k}}$, where $$A_k=\begin{bmatrix} I_k & 0\\
0 & 0 \end{bmatrix}\in \M(n,F)$$ is a rank $k$ matrix.

\begin{lemma}
Let $0 \leq r \leq n$ be an integer and $X \in \mathrm{M}(n,r, F)$. We have
$$
\Theta_{\theta}\left(\left[\begin{array}{cc}
1 & X \\
0 & 1
\end{array}\right]\right)
=(-1)^{2n-1}(q;q)_{2n-r-1}
$$
\end{lemma}

\begin{proof} The proof follows from Theorem~\ref{character value of cuspidal representation (Dipendra)} above and rewriting the character values using the $q$-Pochhammer symbol.  
\end{proof}

\begin{lemma} \label{simplification}
We have 
\[ \sum_{n \in N}\Theta_{\theta}(n)\overline{\psi_{A_k}(n)}= (-1)^{2n-1}\sum_{r=0}^{n}(q;q)_{2n-r-1}(f_{n,r,k}^{0}-f_{n,r,k}^{1}).\]
\end{lemma}

\begin{proof}
Consider
\begin{align*}
\sum_{n \in N}\Theta_{\theta}(n)\overline{\psi_{A_k}(n)}&= \sum_{X \in \M(n,F)} \Theta_{\theta} \left(\left[\begin{array}{cc}
1 & X \\
0 & 1
\end{array}\right]\right) \overline{\psi_{A_k}(X)}\\
&= \sum_{r=0}^{n}\sum_{X \in \M(n,r,F)} \Theta_{\theta} \left(\left[\begin{array}{cc}
1 & X \\
0 & 1
\end{array}\right]\right) \overline{\psi_{A_k}(X)}\\
&= (-1)^{2n-1}\sum_{r=0}^{n}(q;q)_{2n-r-1} \sum_{X \in \M(n,r,F)} \overline{\psi_{A_k}(X)}\\
&=(-1)^{2n-1}\sum_{r=0}^{n}(q;q)_{2n-r-1} \left \{f_{n,r,k}^{0}\psi_0(0)+\sum_{\beta \in F^{\times}}f_{n,r,k}^{\beta} \overline{\psi_0(\beta)} \right \}
\end{align*}
Since $f_{n,r,k}^{\beta}$ is same for all $\beta \in F^{\times}$, we obtain that 
\[\sum_{n \in N}\Theta_{\theta}(n)\overline{\psi_{A_k}(n)}=  (-1)^{2n-1}\sum_{r=0}^{n}(q;q)_{2n-r-1} \left \{ f_{n,r,k}^{0}-f_{n,r,k}^{1} \right\}.\]
\end{proof}

We now prove a $q$-hypergeometric identity that we need to compute the dimension of the twisted Jacquet module $\pi_{N,\psi_{A_k}}$.

\begin{lemma} \label{hyp identity}
We have 
\[\sum_{j=0}^{k} (-1)^j q^{\tbinom{j}{2}} \frac{(q;q)_{k}(a;q)_{k-j}}{(q;q)_j(q;q)_{k-j}}= (-1)^k q^{\tbinom{k}{2}}a^k.\]
\end{lemma}
\begin{proof}
We can write
\begin{equation} \label{eq 1}
\sum_{j=0}^{k} (-1)^j q^{\tbinom{j}{2}} \frac{(q;q)_{k}(a;q)_{k-j}}{(q;q)_j(q;q)_{k-j}}= (q;q)_k \sum_{j=0}^{k} \frac{(-1)^j q^{\tbinom{j}{2}}}{(q;q)_j} \frac{(a;q)_{k-j}}{(q;q)_{k-j}}.
\end{equation}

We recall two classical $q$-identities.
\begin{equation} \label{Euler equation}
(x;q)_{\infty} = \sum_{n=0}^{\infty} \frac{(-1)^n q^{\tbinom{n}{2}}x^n}{(q;q)_n},
\end{equation}

\begin{equation}\label{q-Binomial identity}
\frac{(bx;q)_{\infty}}{(x;q)_{\infty}}=\sum_{n=0}^{\infty} \frac{(b;q)_n}{(q;q)_n}x^n.
\end{equation}
It is enough to consider these two series as formal series in $x$ for this case. The identity \eqref{Euler equation} is due to Euler, whereas the identity \eqref{q-Binomial identity} is called as the `$q$-Binomial identity'.
Plugging $b=a$ in Eq \eqref{q-Binomial identity}, we obtain that 
\[ \sum_{i} \frac{(a;q)_i}{(q;q)_i}x^i=\frac{(ax;q)_{\infty}}{(x;q)_{\infty}}.\]
Hence, Eq \eqref{eq 1} can be written as $(q;q)_k$ times the coefficient of $x^k$ in the following product
\[(x;q)_{\infty}\frac{(ax;q)_{\infty}}{(x;q)_{\infty}}=(ax;q)_{\infty}.\]
We use Eq \eqref{Euler equation} to conclude that the coefficient of $x^k$ equals 
\[\frac{(-1)^k q^{\tbinom{k}{2}} a^k}{(q;q)_k}.\]
Thus, it follows that Eq \eqref{eq 1} becomes 
\[(q;q)_k  \frac{(-1)^k q^{\tbinom{k}{2}} a^k}{(q;q)_k}=(-1)^kq^{\tbinom{k}{2}} a^k.\] 
\end{proof}

\begin{theorem}
Let $\pi=\pi_{\theta}$ be an irreducible cuspidal representation of $\GL(2n,F)$ associated to a regular character of $F_{2n}^{\times}$. For $0 \leq k \leq n$, the dimension of $\pi_{N,\psi_{A_k}}$ is given by 
\[ \dim_{\mathbb{C}}(\pi_{N,\psi_{A_k}})=(-1)^{n-1}(q;q)_{n-1}   q^{\tbinom{k}{2}} (q^k-1)(q^{k+1}-1) \cdots(q^{n-1}-1).\]
\end{theorem}

\begin{proof}
It is easy to see that 
\[\dim_{\mathbb{C}}(\pi_{N,\psi_{A_k}})= \frac{1}{|N|}\sum_{n \in N}\Theta_{\theta}(n)\overline{\psi_{A_k}(n)}.\]
From Lemma \ref{simplification}, it follows that 
\[\dim_{\mathbb{C}}(\pi_{N,\psi_{A_k}})= \frac{(-1)^{2n-1}}{q^{n^2}} \sum_{r=0}^{n} (q,q)_{2n-r-1}(f_{n,r,k}^{0}-f_{n,r,k}^{1}).\]
From Equation \eqref{f(n,r,k)^0-f(n,r,k)^1}, we obtain that
\[(-1)^{2n-1}q^{n^2}\dim_{\mathbb{C}}(\pi_{N,\psi_{A_k}})=  \sum_{r=0}^{n} (q,q)_{2n-r-1} \left \{ \sum_{s=r-k}^{r}(-1)^{r-s}q^{\tbinom{r-s}{2} +ks} \,\Stirling{k}{r-s}_q \,  a(n-k,s,q) \right \}.\]
Then the coefficient of $a(n-k,r,q)$ is given by 
\[q^{kr} (q;q)_{2n-r-k-1} \sum_{j=0}^{k} \frac{ (-1)^j q^{\tbinom{j}{2}}(q;q)_k}{(q;q)_{j}} \frac{(q^{2n-r-k};q)_{k-j}}{(q;q)_{k-j}}. \]
Plugging $a=q^{2n-r-k}$ in Lemma \ref{hyp identity}, the coefficient becomes 
\[q^{kr} (q;q)_{2n-r-k-1}(-1)^kq^{2nk-kr-k^2}q^{\tbinom{k}{2}}.\]
Since $a(n-k,r,q)=0$ for $r> n-k$, it follows that 
\[(-1)^{2n-1+k}\dim_{\mathbb{C}}(\pi_{N,\psi_{A_k}})=q^{-n^2+2nk-\tbinom{k+1}{2}}\sum_{r=0}^{n-k}(q;q)_{2n-r-k-1}a(n-k,r,q).\]
Plugging $t=2n-k-1 \geq 2(n-k)$ in Lemma \ref{Hazan identity}, we obtain that 
\[\dim_{\mathbb{C}}(\pi_{N,\psi_{A_k}})= (-1)^{2n-1+k} q^{-n^2+2nk-\tbinom{k+1}{2}}q^{(n-k)^2} \frac{(q;q)^2_{n-1}}{(q;q)_{k-1}}.\]
Hence, we conclude that 
\[
 \dim_{\mathbb{C}}(\pi_{N,\psi_{A_k}}) =(-1)^{n-1}(q;q)_{n-1} q^{\tbinom{k}{2}}(q^{k}-1)(q^{k+1}-1)\cdots(q^{n-1}-1).\]
\end{proof}

\section{Verification}

Let $0 \leq k \leq n$. In this section, we verify that the dimension of the twisted Jacquet module $\pi_{N,\psi_{A_k}}$ is indeed correct.\\

 For $0\leq k \leq n$, let $\pi[k]$ be the sum of all characters of $N$ inside $\pi$ which lie in the orbit of the character $\psi_{A_{k}}$ under the action of $\GL(n,F) \times \GL(n,F)$. Then, 
\[\dim_{\mathbb{C}}(\pi|_N)= \sum_{k=0}^{n}\dim_{\mathbb{C}}(\pi[k]).\]
Thus, it is enough to show that 
\begin{equation} \label{split of pi|_N}
\sum_{k=0}^{n}a(n,k,q)\dim_{\mathbb{C}}(\pi_{N,\psi_{A_k}})=\dim_{\mathbb{C}}(\pi|_N).
\end{equation}
We have the following formula for $a(n,k,q)$ by Landsberg \cite{landsberg1893ueber}.
\[a(n,k,q)=\frac{(-1)^k(q^{-n};q)^2_k q^{(2nk)-{k \choose 2}}}{(q;q)_k}.\]
Then the $k$-the summand in \eqref{split of pi|_N} can be expressed as follows.
\[ \small
(-1)^{n-1}(q;q)_{n-1}a(n,k,q)q^{{k \choose 2}}(-1)^{n-k}(1-q^k)(1-q^{k+1})\cdots(1-q^{n-1})\]
\[\hspace{2.2 cm}=\frac{(-1)^{2n-1}(q;q)_{n-1}q^{{2nk}}(q^{-n};q)^2_k(q^{k};q)_{n-k}}{(q;q)_k}.\]
Then by applying Lemma \ref{identity for aq^k} for $a=1$, the $k$-th summand is given by:
\[\frac{(-1)^{2n-1}(q;q)_{n-1}q^{{2nk}}(q^{-n};q)^2_k(1;q)_n}{(q;q)_k(1;q)_k}.\]
Substituting $c=1, b=q^{-n}$ and $i=n$ in Eq. \eqref{Chu Vandermode identity}, we conclude that
\begin{align*}
\sum_{k=0}^{n}a(n,k,q)\dim_{\mathbb{C}}(\pi_{N,\psi_{A_k}})&= (-1)^{2n-1}(q;q)_{n-1}(1;q)_n \sum_{k=0}^{n} \frac{q^{{2nk}}(q^{-n};q)^2_k}{(q;q)_k(1;q)_k}\\
&= (-1)^{2n-1}(1;q)_n \frac{(q^n;q)_n}{(1;q)_n}\\
&=(-1)^{2n-1}(q;q)_{n-1}(q^n;q)_n\\
&= (-1)^{2n-1}(q;q)_{2n-1}\\
&=\dim_{\mathbb{C}}(\pi|_N)
\end{align*}

\bibliographystyle{amsplain}
\bibliography{mybib}

\providecommand{\bysame}{\leavevmode\hbox to3em{\hrulefill}\thinspace}
\providecommand{\MR}{\relax\ifhmode\unskip\space\fi MR }
\providecommand{\MRhref}[2]{%
  \href{http://www.ams.org/mathscinet-getitem?mr=#1}{#2}
}
\providecommand{\href}[2]{#2}
\begin{thebibliography}{1}

\bibitem{KumHim[3]}
Kumar Balasubramanian, Abhishek Dangodara, and Himanshi Khurana, \emph{On a
  twisted {J}acquet module of {${\rm GL}(2n)$} over a finite field}, 2022.

\bibitem{KumKriHim}
Kumar Balasubramanian, Krishna Kaipa, and Himanshi Khurana, \emph{A note on the
  cardinality of a certain subset of $n \times n$ matrices over a finite
  field.}, preprint.

\bibitem{KumHimgl6rk2}
Kumar Balasubramanian and Himanshi Khurana, \emph{A certain twisted {J}acquet
  module of {${\rm GL}(6)$} over a finite field: The rank 2 case}, Journal of
  Pure and Applied Algebra \textbf{228} (2024), no.~6, 107614.

\bibitem{gasper2004basic}
George Gasper and Mizan Rahman, \emph{Basic hypergeometric series}, vol.~96,
  Cambridge university press, 2004.

\bibitem{Gel[1]}
S.~I. Gelfand, \emph{Representations of the full linear group over a finite
  field}, Mat. Sb. (N.S.) \textbf{83 (125)} (1970), 15--41. \MR{0272916}

\bibitem{HazGor}
Ofir Gorodetsky and Zahi Hazan, \emph{On certain degenerate {W}hittaker models
  for cuspidal representations of {${\rm GL}_{k\cdot n}(\Bbb F_q)$}}, Math. Z.
  \textbf{291} (2019), no.~1-2, 609--633. \MR{3936084}

\bibitem{landsberg1893ueber}
Georg Landsberg, \emph{Ueber eine anzahlbestimmung und eine damit
  zusammenh{\"a}ngende reihe.},  (1893).

\bibitem{Dip[1]}
Dipendra Prasad, \emph{The space of degenerate whittaker models for general
  linear groups over a finite field}, Internat. Math. Res. Notices (2000),
  no.~11, 579--595. \MR{1763857}

\end{thebibliography}
\end{document}